\documentclass{amsart}

\newtheorem{theorem}{Theorem}[section]
\newtheorem{lemma}[theorem]{Lemma}
\newtheorem{corollary}[theorem]{Corollary}
\newtheorem{proposition}[theorem]{Proposition}

\theoremstyle{definition}
\newtheorem{definition}[theorem]{Definition}

\theoremstyle{remark}
\newtheorem{remark}[theorem]{Remark}

\everymath{\displaystyle}

\numberwithin{equation}{section}

\begin{document}

\title{Gradient estimates on graphs with the $CD\psi(n,-K)$ condition }

\author{Yi Li}
\address{School of Mathematics and Shing-Tung Yau Center of Southeast University, Southeast University, Nanjing, 21189, China}
\curraddr{}
\email{yilicms@gmail.com, yilicms@seu.edu.cn}
\thanks{}

\author{Qianwei Zhang}
\address{School of Mathematics, Southeast University, Nanjing, 21189, China}
\curraddr{}
\email{qianweizhang@seu.edu.cn}
\thanks{}

\subjclass[2010]{Primary }


\date{}

\dedicatory{}

\begin{abstract}
    This paper investigates gradient estimates on graphs satisfying the $CD\psi(n,-K)$ condition with positive constants $n,K$, and concave $C^{1}$ functions $\psi:(0,+\infty)\rightarrow\mathbb{R}$. Our study focuses on gradient estimates for positive solutions of the heat equation $\partial_{t}u=\Delta u$. Additionally, the estimate is extended to a heat-type equation $\partial_{t}u=\Delta u+cu^{\sigma}$, where $\sigma$ is a constant and $c$ is a continuous function defined on $[0,+\infty)$. Furthermore, we utilize these estimates to derive heat kernel bounds and Harnack inequalities.
\end{abstract}

\maketitle
\section{Introduction}

The study of gradient estimates on compact $n$-dimensional  Riemannian manifolds $M$ with lower Ricci curvature bounds $-K(K\geq0)$ has been extensively explored. Li-Yau \cite{liyau} demonstrated that positive solutions $u:M\times[0,+\infty)\rightarrow(0,+\infty)$ of the heat equation 
\begin{equation}
    \partial_{t}u=\Delta u
\label{hq}
\end{equation}
on such manifolds satisfy the inequality
\begin{equation}
    \frac{\left |\nabla u \right |^{2}}{u^{2}}-\alpha\frac{\partial_{t}u}{u}\leq\frac{n\alpha^{2}K}{2(\alpha-1)}+\frac{n\alpha^{2}}{2t}, \ 
    \ \ \forall \ t>0,\  \alpha>1. 
\label{chu}
\end{equation}
In particular, when $K=0$, by taking the limit as $\alpha\rightarrow 1+$, the inequality becomes
\begin{equation*}
    \frac{\left |\nabla u \right |^{2}}{u^{2}}-\frac{\partial_{t}u}{u}\leq\frac{n}{2t},\ \ \ \forall \ t>0.
\end{equation*}
This inequality achieves equality for the heat kernel on $\mathbb{R}^{n}$. Various improvements to the estimate \eqref{chu} have been studied for cases where $K > 0$, yielding gradient estimates such as 
$$\frac{\left |\nabla u \right |^{2}}{u^{2}}-e^{2Kt}\frac{\partial_{t}u}{u}\leq\frac{n}{2t}e^{4Kt}, \ \ \ \forall \ t>0$$
derived by Hamilton \cite{ha},
$$\frac{\left |\nabla u \right |^{2}}{u^{2}}-\left(1+\frac{2K}{3}t\right)\frac{\partial_{t}u}{u}\leq\frac{n}{2t}+\frac{nK}{2}\left(1+\frac{1}{3}Kt\right), \ \ \ \forall\ t>0 $$
obtained by Bakry-Qian \cite{bq},  
$$\frac{\left |\nabla u \right |^{2}}{u^{2}}-\left(1+\frac{\sinh{Kt}\cosh{Kt}-Kt}{\sinh^{2}Kt}\right)\frac{\partial_{t}u}{u}\leq\frac{nK}{2}(1+\coth{Kt}), \ \ \ \forall \ t>0 $$
proved by Li-Xu \cite{lixu}, and 
\begin{equation}
\begin{aligned}
     \frac{\left |\nabla u \right |^{2}}{u^{2}}-\frac{\partial_{t}u}{u}\leq&
    \frac{n}{2t}+\frac{1}{t}\sqrt{2nK(1+Kt)(1+t)}\ {\rm diam}(M)\\
    &+\frac{1}{t}\sqrt{K(1+Kt)(C_{1}+C_{2}K)t}, \ \ \ \forall \ t>0
\label{sharp}
\end{aligned}
\end{equation}
established by Zhang \cite{zq}, where $C_{1}, C_{2}$ are positive constants depending only on the dimension $n$. These estimates involve additional terms that depend on the properties of the manifolds and provide sharper bounds on the growth of gradients. 
Furthermore, the Li-Yau estimates have been extended to heat-type equations, for example,  
$$(\Delta-\partial_{t})u+hu^{\alpha}=0, $$
where $\alpha$ is a positive constant and $h:M\times[0,\infty)\rightarrow\mathbb{R}$ is $C^{2}$ in the first variable and $C^{1}$ in the second variable \cite{jayu}, and 
$$(\Delta-\partial_{t})u+au\log{u}=0,$$
where $a\ne 0$ is a constant \cite{cao}. 

However, the discrete analog of gradient estimates on graphs presents challenges due to the failure of the chain rule and the absence of a suitable notion of curvature. To address these challenges, Bauer et al.\cite{sqrt} introduced the equation 
$$\Delta f=2\sqrt{f}\Delta \sqrt{f}+2\Gamma\left(\sqrt{f}\right), $$
for all $f:V\rightarrow[0,+\infty)$, as an alternative to the chain law and proposed the $CDE(n,K)$ condition  as a replacement for the lower Ricci curvature bound (cf. Section \ref{2}). M\"unch defined a scaling invariant Laplacian operator $\Delta^{\psi}$ and a second gradient operator $\Gamma_{2}^{\psi}$ for any concave function $\psi$ in order to establish the $CD\psi(n,K)$ condition (cf. Section \ref{2}). Additionally, an identity was derived
\begin{equation}
    (\Delta-\partial_{t})\left(-u\Delta^{\psi}u\right)=2u\Gamma_{2}^{\psi}\left(u\right)
\label{chaa}
\end{equation}
which holds for all solutions of the heat equation \eqref{hq} and eliminates the need for utilizing the chain rule (cf. Subsection \ref{3.2}). 

Under the aforementioned framework, Bauer et al.\cite{sqrt} proved that on a (finite or infinite) graph with the $CDE(n,-K)$ for some $K>0$, the following inequality should hold for all positive solutions of the heat equation \eqref{hq} 
$$ \frac{\Gamma\left(\sqrt{u}\right)}{u}-\alpha\frac{\partial_{t}\left(\sqrt{u}\right)}{\sqrt{u}}\leq \frac{n\alpha^{2}}{2t}+\frac{Kn\alpha^{2}}{\alpha-1}, \ \ \ \forall\ \alpha>1. $$  
Regarding Hamilton's estimate, Bakry-Qian's estimate, and Li-Xu's estimate on graphs, separate studies have been conducted by Wang et al. \cite{w} and Lv et al. \cite{wlf} to investigate the $CDE$ condition and the $CD\psi$ condition. However, Zhang's estimate on graphs has not been established yet. Furthermore, although there have been numerous studies on gradient estimates for heat-type equations on manifolds, they are not as prevalent as those conducted on graphs. In a related work by M\"unch \cite{gama}, a gradient estimate was derived for finite graphs under the $CDE(n,0)$ condition. This estimate applies to all positive solutions of the heat-type equation 
$$\partial_{t}u=\Delta u+\Gamma u. $$ 
Moreover, Sun \cite{sun} obtained a gradient estimate for finite graphs under the $CD\psi(n,0)$ condition. This estimate applies to all positive solutions of the heat-type equation  
$$\partial_{t}u=\Delta u+cu^{\sigma},$$
where $c,\sigma$ are constants.

In this paper, our aim is to establish estimates analogous to Zhang's estimate \eqref{sharp} for graphs with negative lower bounds on curvature. We begin our investigation by examining estimates for positive solutions of the heat equation \eqref{hq}. Under the $CDE(n,-K)$ condition for a given $K>0$, we present a global estimate for finite graphs and a local estimate for infinite graphs as follows (the definitions of the parameters $D_{\mu},D_{w}$ refer to Section \ref{2}). 
\begin{theorem}
    Let $G=(V,E)$ be a finite graph satisfying the $CDE(n,-K)$ for some $K>0$ and $u$ be a positive solution of heat equation \eqref{hq} on V. Then 
    $$\frac{\Gamma(\sqrt{u})}{u}-\frac{\partial_{t}(\sqrt{u})}{\sqrt{u}}\leq \frac{n}{2t}+\sqrt{\frac{1}{2}nKD_{\mu}(D_{w}+1)},\ \ \ \forall \ t>0. $$
\end{theorem}
\begin{theorem}
    Let $G=(V,E)$ be a (finite or infinite) graph satisfying the $CDE(n,-K)$ for some $K>0$. Fix $R>0$ and $x_{0}\in V$. Suppose $u$ is a positive solution of the heat equation \eqref{hq} on the ball $B(x_{0},2R)$. Then, in the ball $B(x_{0},R)$, 
    $$\frac{\Gamma(\sqrt{u})}{u}-\frac{\partial_{t}(\sqrt{u})}{\sqrt{u}}\leq \frac{n}{2t}+\sqrt{\frac{1}{2}nKD_{\mu}(D_{w}+1)}+\frac{nD_{\mu}(1+D_{w})}{R},\ \ \ \forall\ t>0.$$
\end{theorem}
Additionally, we consider the $CD\psi(n,-K)$ condition for concave $C^{1}$ functions $\psi$ and derive corresponding gradient estimates. 
\begin{theorem}
    Let $G=(V,E)$ be a finite graph satisfying the $CD\psi(n,-K)$ for some $K>0$. Let $\psi:(0,+\infty)\rightarrow \mathbb{R}$ be a concave $C^{1}$ function with $\psi,\ \psi^{\prime}>0$. Suppose $u$ is a positive solution of the heat equation \eqref{hq} on $V$. Then
    $$\Gamma^{\psi}(u)-{\psi}^{\prime}(1)\frac{\partial_{t}u}{u} \leq \frac{n}{2t}+\sqrt{nKC},\ \ \ \forall\ t>0,$$
    where $C=D_{\mu}\left[\psi^{\prime}(1)\left(\psi^{-1}(\psi(1)D_{w})-1\right)+\psi(1)\right].$
\end{theorem}     
\begin{theorem}
    Let $G=(V,E)$ be a (finite or infinite) graph with the $CD\psi(n,-K)$ for some $K>0$. Let $\psi:(0,+\infty)\rightarrow \mathbb{R}$ be a concave, $C^{1}$ function with $\psi,\ \psi^{\prime}>0$. Fix $R>0$ and $x_{0}\in V$. Suppose $u$ is a positive solution of the heat equation \eqref{hq} on the ball $B(x_{0},2R)$.   
    Then, in the ball $B(x_{0},R)$,  
    $$\Gamma^{\psi}(u)-{\psi}^{\prime}(1)\frac{\partial_{t}u}{u}\leq \frac{n}{2t}+\sqrt{nKC}+\frac{n D_{\mu}\left[\psi(1)+D_{w}\right]}{R},\ \ \ \forall \ t>0,$$
    where $C=D_{\mu}\left[\psi^{\prime}(1)\left(\psi^{-1}(\psi(1)D_{w})-1\right)+\psi(1)\right].$
\end{theorem}

Furthermore, we extend the estimate to the heat-type equation
\begin{equation}
    \partial_{t}u=\Delta u+cu^{\sigma},
\label{yama}
\end{equation}
where $\sigma$ is a constant and $c:[0,+\infty)\rightarrow\mathbb{R}$ is a continuous function satisfying $c\geq0, \sigma\leq1$ or $c\leq0, \sigma\geq1$.
In this case, we find that the inequality analogous to \eqref{chaa}
\begin{equation}
    (\Delta-\partial_{t})(-u\Delta^{\psi}u)\geq2u\Gamma_{2}^{\psi}(u)+cu^{\sigma}\Delta^{\psi}u
\label{cha}
\end{equation}
holds for all positive solutions of the equation \eqref{yama}. By utilizing the inequality \eqref{cha} as a replacement for the chain rule, we obtain the estimate in the $CD\psi(n,-K)$ condition for some $C^{1}$, concave functions $\psi$. 
\begin{theorem}
    Let $G=(V,E)$ be a finite graph satisfying the $CD\psi(n,-K)$ for some $K>0$, where $\psi:(0,+\infty)\rightarrow \mathbb{R}$ is a $C^{1}$, concave function with $\psi^{\prime}(1)=0$. Suppose $u$ is a positive solution of the equation \eqref{yama} on $V$.  
    Then 
    $$\Gamma^{\psi}(u) \leq \frac{n}{2t}+Kn,\ \ \ \forall\ t>0. $$
\end{theorem}    

The remaining sections of the paper further develop the concepts introduced in the introduction, focusing on the derived gradient estimates and their applications in heat kernel bounds and Harnack inequalities.

We structure the paper as follows. Section 2 presents the essential background and definitions, including the $CDE(n, K)$ condition and the $CD\psi(n,-K)$ condition. In Sections 3 and 4, we establish estimates for positive solutions of the heat equation \eqref{hq} and the nonlinear parabolic equation \eqref{yama}, respectively.

\section{Notations}
\label{2}
Let $G=(V, E)$ be a graph, where $V$ denotes the vertex set and $E$ denotes the edge set.
Throughout this paper, all graphs are assumed to be connected. For any edge $xy\in E$, we assume that its weight $w_{xy}> 0$ (which may be asymmetric). 
For any $x\in V$, 
$${\rm deg}(x):=\sum_{y\sim x}w_{xy}<+\infty,$$
i.e., the graph is locally finite. Given a measure $\mu :V\rightarrow \mathbb{R}$, the volume of the subset $\Omega\subset V$ is defined by 
$${\rm vol}(\Omega)=\sum_{x\in \Omega}\mu(x). $$
Moreover, we define
\begin{equation*}
    D_{\mu}:=\sup_{x\in V}\frac{{\rm deg}(x)}{\mu(x)},\ \ 
    D_{w}:=\sup_{x\sim y}\frac{{\rm deg}(x)}{w_{xy}}. 
\end{equation*}

\begin{definition}
    For any function $f :V\rightarrow \mathbb{R}$, the Laplacian of $f$ is defined by
    \begin{equation*}
      \Delta f(x):=\frac{1}{\mu (x)}\sum_{y\sim x}w_{xy}\left[f(y)-f(x)\right],
    \end{equation*}
where $y\sim x$ means $xy\in E$.
\end{definition}
To simplify the calculation, we introduce the notation, for a vertex $x \in V$, 
$$\widetilde{\sum_{y\sim x}}f(y):=\frac{1}{\mu(x)}\sum_{y\sim x}w_{xy}f(y). $$
\begin{definition}
The gradient form is defined as 
$$2\Gamma (f,g):=\Delta(fg)-f\Delta g-g\Delta f, $$
for functions $f,\ g: V\rightarrow \mathbb{R}$.
Similarly, the second gradient form is defined as 
$$2\Gamma_{2} (f,g):=\Delta\Gamma(fg)-\Gamma(f,\Delta g)-\Gamma(\Delta f,g).$$
Write $ \Gamma (f):= \Gamma (f,f), \Gamma_{2} (f):= \Gamma_{2} (f,f)$.
\end{definition}
In \cite{sqrt}, the authors introduced the $CDE(n,K)$ condition below as an alternative to the lower Ricci curvature bound on graphs.

\begin{definition}
    For $n>0, K\in \mathbb{R}$, we say that a graph $G$ satisfies the $CDE(n,K)$ condition, if for any positive function $f :V\rightarrow (0,+\infty)$, we have 
    $$\widetilde{\Gamma}_{2}(f)\geq \frac{1}{n}(\Delta f)^{2}+K\Gamma(f)                            , $$
    where
    $$\widetilde{\Gamma}_{2}(f):=\Gamma_{2}(f)-\Gamma\left(f,\frac{\Gamma(f)}{f}\right). $$
\end{definition}
\begin{remark}
    Through simple calculation, one has
    $$\widetilde{\Gamma}_{2}(f):=\frac{1}{2}\Delta\Gamma(f)-\Gamma\left(f,\frac{\Delta(f^{2})}{2f}\right). $$
\end{remark}

Bauer et al. derived gradient estimates within the framework of the $CDE(n,K)$ condition \cite{sqrt} by considering the function $f=\sqrt{u}$, where $u$ represents a positive solution of the heat equation \eqref{hq}. However, M\"unch extended these findings by considering the function $f=\psi(u)$, where $\psi$ is a concave function. The following definitions, as presented in \cite{psi}, are relevant to this generalization.

\begin{definition}
   Let $\psi:(0,+\infty)\rightarrow \mathbb{R}$ be a function. Then the $\psi-$Laplacian $\Delta^{\psi}$ is defined by 
   $$\Delta^{\psi}f(x):=\Delta \left[\psi\left(\frac{f}{f(x)}\right)\right](x), $$
   for all $f:V\rightarrow (0,+\infty) $. 
\end{definition}
\begin{definition}
   Let $\psi:(0,+\infty)\rightarrow \mathbb{R}$ be a $C^{1}$ function. We define the $\psi-$gradient $\Gamma^{\psi}$ by 
   $$\Gamma^{\psi}:=\Delta^{\overline{\psi}} , $$
   here $\overline{\psi}:(0,+\infty)\rightarrow\mathbb{R}$ is defined by 
   $$\overline{\psi}(s)=\psi^{\prime}(1)(s-1)-\left[\psi(s)-\psi(1)\right].$$
\end{definition}

\begin{remark}
   In \cite{psi}, M\"unch pointed out that if $\psi:(0,+\infty)\rightarrow \mathbb{R}$ is a $C^{1}$, concave function, then $\Gamma^{\psi}f\geq0$ for all $f:V\rightarrow (0,+\infty)$.  
\end{remark}

\begin{definition}
   Let $\psi:(0,+\infty)\rightarrow \mathbb{R}$ be a $C^{1}$ function. We define the second $\psi-$gradient $\Gamma_{2}^{\psi}$ by 
   $$2\Gamma_{2}^{\psi}:=\Omega^{\psi}+\frac{\Delta f \Delta^{\psi}}{f}-\frac{\Delta(f\Delta^{\psi}f)}{f} , $$
   for any positive function $f:V\rightarrow (0,+\infty) $ and here $\Omega^{\psi}$ is defined by 
   $$(\Omega^{\psi}f)(x):=\Delta\left[\psi^{\prime}\left(\frac{f}{f(x)}\right)\cdot\frac{f}{f(x)}\left(\frac{\Delta f}{f}-\frac{\Delta f(x)}{f(x)}\right)\right](x).$$
\end{definition}

\begin{definition}
    Let $\psi:(0,+\infty)\rightarrow \mathbb{R}$ be a $C^{1}$ function. For $n>0, K\in \mathbb{R}$, we say that a graph $G$ satisfies the $CD\psi(n,K)$ condition, if for any function $f :V\rightarrow (0,+\infty)$, we have 
    $$\Gamma_{2}^{\psi}(f)\geq \frac{1}{n}(\Delta^{\psi} f)^{2}+K\Gamma^{\psi}(f). $$
\end{definition}

By proposing the Harnack constant below, one can obtain the Harnack inequality with the $CD\psi$ condition. 
\begin{definition}
    Let $\psi:(0,+\infty)\rightarrow \mathbb{R}$ be a $C^{1}$ function. We define the Harnack constant $H_{\psi}$ of $\psi$ as
    $$H_{\psi}:=\sup_{x>1}\frac{\log^{2} x}{\overline{\psi}(x)}.$$
\end{definition}

\section{Gradient estimates for the heat equation }
\label{3}
In this section, we examine estimate for the heat equation \eqref{hq} under the $CDE(n,-K)$ condition and the $CD\psi(n,-K)$ condition for some $K>0$ separately. We then apply this estimate to obtain heat kernel estimates. Let us recall the operator $\mathcal{L}=\Delta-\partial_{t}.$ Correspondingly, the following maximum principle can be widely used in the proof of gradient estimates. 
\begin{lemma}\cite{sqrt}
    Let $G=(V,E)$ be a graph, and let $g, F: V\times [0,T] \rightarrow \mathbb{R}$ be functions. Suppose that $g\geq0$, and $F$ has a local maximum at $(x^{*},t^{*}) \in V\times(0,T]$. Then 
    $$\mathcal{L}(gF)(x^{*},t^{*})\leq(\mathcal{L}g)F(x^{*},t^{*}). $$
\label{max}
\end{lemma}

\subsection{The $CDE$ condition}
Firstly we recall the lemma in \cite{sqrt}. 
\begin{lemma}\cite{sqrt}
    Suppose $f: V\rightarrow (0,+\infty)$ and $(\Delta f)(x)<0$ at some vertex $x$. Then 
    $$\widetilde{\sum_{y\sim x}}f^{2}(y)<D_{\mu}D_{w}f^{2}(x).$$
\label{sqrtbound}
\end{lemma}
Now we begin proving the gradient estimate on finite graphs. 
\begin{theorem}
     Let $G=(V,E)$ be a finite graph satisfying the $CDE(n,-K)$ for some $K>0$ and $u$ be a positive solution of heat equation \eqref{hq} on V. Then 
    $$\frac{\Gamma(\sqrt{u})}{u}-\frac{\partial_{t}(\sqrt{u})}{\sqrt{u}}\leq \frac{n}{2t}+\sqrt{\frac{1}{2}nKD_{\mu}(D_{w}+1)},\ \ \ \forall \ t>0. $$
\label{sf}
\end{theorem}

\begin{proof}
    Let $F=t\left(\frac{2\Gamma(\sqrt{u})}{u}-\frac{2\partial_{t}(\sqrt{u})}{\sqrt{u}}\right)$. Since 
    \begin{equation}
        \Delta u=2\sqrt{u}\Delta \sqrt{u}+2\Gamma(\sqrt{u}),
    \label{chain}
    \end{equation}
    we have
    $F=t \frac{-2\Delta \sqrt{u}}{\sqrt{u}}.$
    Fix $T>0$. Let $(x^{*},t^{*})$ be a maximum point of $F$ in $V\times [0,T]$. It suffices to prove that 
    $$F(x^{*},t^{*})\leq n+2T\sqrt{\frac{1}{2}nKD_{\mu}(D_{w}+1)}.$$
    We may assume $F(x^{*},t^{*})>0$. Hence $t^{*}>0$ and $\Delta\sqrt{u}(x^{*},t^{*})<0$. In the subsequent analysis, all computations are considered to be carried out at the point $(x^{*},t^{*})$. We apply Lemma \ref{max} with $g=u$. This gives
    \begin{equation*}
    \begin{aligned}
        0=\mathcal{L}(u)F\geq\mathcal{L}(uF)
        &=t^{*}\mathcal{L}\left[2\Gamma(\sqrt{u})-\Delta u\right]-\frac{Fu}{t^{*}}\\
        &=t^{*}\left[2\Delta\Gamma(\sqrt{u})-4\Gamma\left(\sqrt{u}, \frac{\Delta u}{2\sqrt{u}}\right)\right]-\frac{Fu}{t^{*}}\\
        &=4t^{*}\widetilde{\Gamma}_{2}(\sqrt{u})-\frac{Fu}{t^{*}}\\
        &\geq \frac{4t^{*}}{n}(\Delta \sqrt{u})^{2}-4t^{*}K\Gamma(\sqrt{u})-\frac{Fu}{t^{*}}\\
        &=\frac{u}{nt^{*}}F^{2}-4t^{*}K\Gamma(\sqrt{u})-\frac{Fu}{t^{*}}
    \end{aligned}    
    \end{equation*}
    where we use the $CDE(n,-K)$ condition in the penultimate step and the equation $F=t^{*}\frac{-2\Delta \sqrt{u}}{\sqrt{u}}$ in the last step. 
    Hence, we have 
    $$F\leq n+2t^{*}\sqrt{nK}\sqrt{\frac{\Gamma(\sqrt{u})}{u}}.$$
    Since the fact $\Delta \sqrt{u}(x^{*},t^{*})<0$, one can apply Lemma \ref{sqrtbound} which yields
    \begin{equation*}
    \begin{aligned}
        2\Gamma(\sqrt{u})(x^{*},t^{*})=\widetilde{\sum_{y\sim x^{*}}}\left[\sqrt{u}(y,t^{*})-\sqrt{u}(x^{*},t^{*})\right]^{2}
        &\leq\widetilde{\sum_{y\sim x^{*}}}\left[u(y,t^{*})+u(x^{*},t^{*})\right]\\
        &\leq D_{\mu}(D_{w}+1)u(x^{*},t^{*}).
    \end{aligned}    
    \end{equation*}
    Therefore, $$F\leq n+2t^{*}\sqrt{\frac{1}{2}nKD_{\mu}(D_{w}+1)}. $$ Our desired result follows. 
\end{proof}

Then we consider a local gradient estimate on graphs.
\begin{theorem}
    Let $G=(V,E)$ be a (finite or infinite) graph with the $CDE(n,-K)$ for some $K>0$. Fix $R>0$ and $x_{0}\in V$. Suppose $u$ is a positive solution of the heat equation \eqref{hq} on the ball $B(x_{0},2R)$. Then, in the ball $B(x_{0},R)$, 
    $$\frac{\Gamma(\sqrt{u})}{u}-\frac{\partial_{t}(\sqrt{u})}{\sqrt{u}}\leq \frac{n}{2t}+\sqrt{\frac{1}{2}nKD_{\mu}(D_{w}+1)}+\frac{nD_{\mu}(1+D_{w})}{R},\ \ \ \forall\ t>0. $$
\label{inf}
\end{theorem}

\begin{proof}
    Let us define a cut-off function $\phi: V \rightarrow \mathbb{R}$ as 
    $$\phi(v)=
    \begin{cases}
    0,& d(v,x_{0})>2R\\ 
    \frac{2R-d(v,x_{0})}{R},& 2R \geq d(v,x_{0})\geq R\\
    1,& R>d(v,x_{0})
    \end{cases}. $$
    Let $F=t\phi\left(\frac{2\Gamma(\sqrt{u})}{u}-\frac{2\partial_{t}(\sqrt{u})}{\sqrt{u}}\right)=t\phi \frac{-2\Delta \sqrt{u}}{\sqrt{u}}$. Fix $T>0$. Let $(x^{*},t^{*})$ be a maximum point of $F$ in $V\times [0,T]$. It suffices to prove that 
    $$F(x^{*},t^{*})\leq n+2T\sqrt{\frac{1}{2}nKD_{\mu}(D_{w}+1)}+2T\frac{nD_{\mu}(1+D_{w})}{R}.$$
    We may assume $F(x^{*},t^{*})>0$. Hence $t^{*}>0$ and $\Delta\sqrt{u}(x^{*},t^{*})<0$. In the subsequent analysis, all computations are considered to be carried out at the point $(x^{*},t^{*})$. As showed in \cite{sqrt}, the positivity of $u$ implies 
    $$\frac{-\Delta \sqrt{u}}{\sqrt{u}}\leq D_{\mu}, $$
    so we may assume $\phi(x^{*})\geq \frac{2}{R}$ and $ R\geq2$. 
    Hence, $\phi$ does not vanish in the neighborhood of $x^{*}$. Now we apply Lemma \ref{max} with $g=\frac{u}{\phi}$. This gives
    \begin{equation*}
    \begin{aligned}
        \mathcal{L}\left(\frac{u}{\phi}\right)F\geq\mathcal{L}\left(\frac{u}{\phi}F\right)&=4t^{*}\widetilde{\Gamma}_{2}(\sqrt{u})-\frac{Fu}{t^{*}\phi}\\
        &\geq \frac{4t^{*}}{n}(\Delta \sqrt{u})^{2}-4t^{*}K\Gamma(\sqrt{u})-\frac{Fu}{t^{*}\phi}\\
        &=\frac{u}{nt^{*}\phi^{2}}F^{2}-4t^{*}K\Gamma(\sqrt{u})-\frac{Fu}{t^{*}\phi}
    \end{aligned}    
    \end{equation*}
    where we use the $CDE(n,-K)$ condition in the penultimate step and the equation $F=t^{*}\phi \frac{-2\Delta \sqrt{u}}{\sqrt{u}}$ in the last step. 
    Hence, we have    
    $$F\leq n+2t^{*}\sqrt{nK}\sqrt{\frac{\Gamma(\sqrt{u})}{u}}+2t^{*}n\left[\frac{\phi^{2}}{2u}\mathcal{L}\left(\frac{u}{\phi}\right)\right].$$
    According to \cite{sqrt}, we have $$\frac{\phi^{2}}{2u}\mathcal{L}\left(\frac{u}{\phi}\right)<\frac{D_{\mu}D_{w}}{R}.$$
    By the proof of Theorem \ref{sf}, we obtain
    $$F\leq n+2t^{*}\sqrt{\frac{1}{2}nKD_{\mu}(D_{w}+1)}+2t^{*}\frac{nD_{\mu}D_{w}}{R}. $$ 
    Our desired result follows. 
\end{proof}

\begin{corollary}
    Let $G=(V,E)$ be an infinite graph satisfying the $CDE(n,-K)$ for some $K>0$ with $D_{\mu},D_{w}<\infty$. Suppose u is a positive solution of the heat equation \eqref{hq} on $V$. Then 
    $$\frac{\Gamma(\sqrt{u})}{u}-\frac{\partial_{t}(\sqrt{u})}{\sqrt{u}}\leq \frac{n}{2t}+\sqrt{\frac{1}{2}nKD_{\mu}(D_{w}+1)},\ \ \ \forall \ t>0.$$
\end{corollary}

\subsection{The $CD\psi$ condition}
\label{3.2}
Now we prove the estimate under the $CD\psi$ condition. Let us recall two key lemmas in \cite{psi}. 
\begin{lemma}\cite{psi}
    Let $\psi:(0,+\infty)\rightarrow \mathbb{R}$ be a $C^{1}$ function. Then one has 
    $$-\Delta^{\psi}f=\Gamma^{\psi}(f)-{\psi}^{\prime}(1)\frac{\Delta f}{f},$$
    for all $f:V\rightarrow(0,+\infty)$.
\label{pchain}
\end{lemma}

\begin{lemma}\cite{psi}
    Let $\psi:(0,+\infty)\rightarrow \mathbb{R}$ be a $C^{1}$ function. Suppose u is a positive solution of the heat equation \eqref{hq} on $V$. Then 
    $$\mathcal{L}(-u\Delta^{\psi}u)=2u\Gamma_{2}^{\psi}(u). $$
\label{L}
\end{lemma}

To enhance our understanding of the forthcoming estimate, we shall initially present two general gradient estimates under the $CD\psi(n,-K)$ condition.

\begin{theorem}
    Let $G=(V,E)$ be a finite graph satisfying the $CD\psi(n,-K)$ for some $K>0$ where $\psi:(0,+\infty)\rightarrow \mathbb{R}$ is a $C^{1}$, concave function. Suppose $u$ is a positive solution of the heat equation \eqref{hq} on $V$. 
    Fix $0<\alpha<1$. Then 
    \begin{equation*}
    (1-\alpha)\Gamma^{\psi}u-{\psi}^{\prime}(1)\frac{\partial_{t}u}{u}\leq \frac{n}{(1-\alpha)2t}+\frac{Kn}{\alpha},\ \ \ \forall \ t>0. 
    \end{equation*}
\label{p}
\end{theorem}   

\begin{remark}
    When $\psi=\sqrt{\cdot}$, we obtain $\Gamma^{\psi}u=\frac{\Gamma(\sqrt{u})}{u}$ and Theorem \ref{p} becomes Theorem 4.10 in \cite{sqrt}. 
\end{remark} 

\begin{proof}
    Let $F=t\left[(1-\alpha)\Gamma^{\psi}u-{\psi}^{\prime}(1)\frac{\Delta u}{u}\right]$. By Lemma \ref{pchain}, we obtain 
    $$F=t\left[-\alpha{\psi}^{\prime}(1)\frac{\Delta u}{u}-(1-\alpha)\Delta^{\psi}u\right]=-t\left[\Delta^{\psi}u+\alpha\Gamma^{\psi}(u)\right].$$
    Fix $T>0$. Let $(x^{*},t^{*})$ be a maximum point of $F$ in $V\times [0,T]$. It suffices to prove that 
    $$F(x^{*},t^{*})\leq \frac{n}{2(1-\alpha)}+T\frac{Kn}{\alpha}.$$ 
    We may assume $F(x^{*},t^{*})>0$. Hence $t^{*}>0$. In the subsequent analysis, all computations are considered to be carried out at the point $(x^{*},t^{*})$. We apply Lemma \ref{max} with $g=\frac{u}{t^{*}}$. This gives
    \begin{equation*}
    \begin{aligned}
        \frac{u}{(t^{*})^{2}}F=\mathcal{L}\left(\frac{u}{t^{*}}\right)F\geq\mathcal{L}\left(\frac{u}{t^{*}}F\right)&=(1-\alpha)\mathcal{L}(-u \Delta^{\psi}u)\\
        &=2(1-\alpha)u\Gamma^{\psi}_{2}(u)\\
        &\geq 2(1-\alpha)u\left[\frac{1}{n}(\Delta^{\psi}u)^{2}-K\Gamma^{\psi}(u)\right]
    \end{aligned}    
    \end{equation*}
    where we use Lemma \ref{L} in the penultimate step and the $CD\psi(n,-K)$ condition in the last step. 
    In view of $F=-t^{*}\left[\Delta^{\psi}u+\alpha\Gamma^{\psi}(u)\right]$, we have 
    $$\frac{F}{1-\alpha}\geq 2\left[\frac{1}{n}\left(F+\alpha t^{*}\Gamma^{\psi}(u)\right)^{2}-K(t^{*})^{2}\Gamma^{\psi}(u)\right]. $$
    Let us denote $H=t^{*}\Gamma^{\psi}(u)$. Then we obtain
    $$\frac{F}{1-\alpha}\geq 2\left[\frac{1}{n}(F+\alpha H)^{2}-t^{*}KH\right]. $$
    Since $\psi$ is concave, one has $H\geq0$. After expanding $(F+\alpha H)^{2}$ we throw away the $FH$ term, and use  $\alpha^{2}H^{2}$ to bound the last term $t^{*}KH$. So we have 
    $$\frac{F}{1-\alpha}\geq 2\left[\frac{1}{n}F^{2}-\frac{K^{2}n}{4\alpha^{2}}(t^{*})^{2}\right], $$
    which implies
    $$F\leq \frac{n}{2(1-\alpha)}+\frac{Knt^{*}}{\alpha}. $$
    And our desired result follows. 
\end{proof}  

\begin{theorem}
    Let $G=(V,E)$ be a (finite or infinite) graph satisfying the $CD\psi(n,-K)$ for some $K>0$ where $\psi:(0,+\infty)\rightarrow \mathbb{R}$ is a $C^{1}$, concave function with $\psi>0$. Fix $R>0$, $x_{0}\in V$ and $0\leq\alpha<1$. Suppose $u$ is a positive solution of the heat equation \eqref{hq} on the ball $B(x_{0},2R)$. Then, in the ball $B(x_{0},R)$,  
    $$(1-\alpha)\Gamma^{\psi}(u)-{\psi}^{\prime}(1)\frac{\partial_{t}u}{u}\leq \frac{n}{2(1-\alpha)t}+\frac{Kn}{\alpha}+\frac{n D_{\mu}\left[\psi(1)+D_{w}\right]}{(1-\alpha)R},\ \ \ \forall \ t>0. $$
\label{alpha}
\end{theorem}   

\begin{remark}
    When $\psi=\sqrt{\cdot}$, we obtain $\Gamma^{\psi}(u)=\frac{\Gamma(\sqrt{u})}{u}$ and Theorem \ref{alpha} becomes Theorem 4.20 in \cite{sqrt}. 
\end{remark} 

\begin{proof}
    Let us define a cut-off function $\phi: V \rightarrow \mathbb{R}$ as 
    $$\phi(v)=
    \begin{cases}
    0,& d(v,x_{0})>2R\\ 
    \frac{2R-d(v,x_{0})}{R},& 2R \geq d(v,x_{0})\geq R\\
    1,& R>d(v,x_{0})
    \end{cases}. $$
    Let $F=t\phi\left[(1-\alpha)\Gamma^{\psi}u-{\psi}^{\prime}(1)\frac{\Delta u}{u}\right]$. By Lemma \ref{pchain}, we obtain 
    $$F=t\phi\left[-\alpha{\psi}^{\prime}(1)\frac{\Delta u}{u}-(1-\alpha)\Delta^{\psi}u\right]=-t\phi\left[\Delta^{\psi}u+\alpha\Gamma^{\psi}(u)\right].$$
    Fix $T>0$. Let $(x^{*},t^{*})$ be a maximum point of $F$ in $V\times [0,T]$. It suffices to prove that 
    $$F(x^{*},t^{*})\leq \frac{n}{2(1-\alpha)}+T\frac{Kn}{\alpha}+T\frac{n D_{\mu}\left[\psi(1)+D_{w}\right]}{(1-\alpha)R}.$$ 
    We may assume $F(x^{*},t^{*})>0$. Hence $t^{*}>0$. In the subsequent analysis, all computations are considered to be carried out at the point $(x^{*},t^{*})$. The positivity of $\psi$ implies that for any $x\in V$, 
    $$-\Delta^{\psi}f(x)=-\Delta\psi\left(\frac{f}{f(x)}\right)(x)=\widetilde{\sum_{y\sim x}}\left[\psi(1)-\psi\left(\frac{f(y)}{f(x)}\right)\right]\leq \psi(1)D_{\mu}. $$
    So we may assume $\phi(x^{*})\geq \frac{2}{R}$ and $R\geq2$. Hence $\phi$ does not vanish in the neighborhood of $x^{*}$. Now we apply lemma with $g=\frac{u}{\phi t^{*}}$. This gives
    \begin{equation*}
    \begin{aligned}
        \left [t^{*}\mathcal{L}\left(\frac{u}{\phi}\right)+\frac{u}{\phi}\right]\frac{F}{(t^{*})^{2}}=\mathcal{L}\left(\frac{u}{\phi t^{*}}\right)F&\geq\mathcal{L}\left(\frac{u}{\phi t^{*}}F\right)\\
        &=2u(1-\alpha){\Gamma}_{2}^{\psi}(u)\\
        &\geq 2u(1-\alpha)\left[\frac{1}{n}(\Delta^{\psi}u)^{2}-K\Gamma^{\psi}(u)\right]
    \end{aligned}    
    \end{equation*}
    where we used the $CD\psi(n,-K)$ condition in the last step. Simplifying the above inequality with $F=-t^{*}\phi \left[\Delta^{\psi}u+\alpha\Gamma^{\psi}(u)\right]$, we have 
    $$\left[\frac{t^{*}\phi^{2}}{2u}\mathcal{L}\left(\frac{u}{\phi}\right)+\frac{\phi}{2}\right]\frac{F}{1-\alpha}\geq \frac{1}{n}\left[F+\alpha t^{*}\phi\Gamma^{\psi}(u)\right]^{2}-K(t^{*}\phi)^{2}\Gamma^{\psi}(u). $$
    Let us denote $H=t^{*}\phi\Gamma^{\psi}(u)$. Then we obtain
    $$\left[\frac{t^{*}\phi^{2}}{2u}\mathcal{L}\left(\frac{u}{\phi}\right)+\frac{\phi}{2}\right]\frac{F}{1-\alpha}\geq \frac{1}{n}(F+\alpha H)^{2}-t^{*}\phi KH. $$
    After expanding $(F+\alpha H)^{2}$ we throw away the $FH$ term, and use  $\alpha^{2}H^{2}$ to bound the last term $t^{*}\phi KH$. So we have 
    $$\left[\frac{t^{*}\phi^{2}}{2u}\mathcal{L}\left(\frac{u}{\phi}\right)+\frac{\phi}{2}\right]\frac{F}{1-\alpha}\geq \frac{1}{n}F^{2}-\frac{K^{2}n}{4\alpha^{2}}(t^{*}\phi)^{2},$$
    which implies
    $$F\leq \frac{n}{2(1-\alpha)}+\frac{Knt^{*}}{\alpha}+t^{*}\frac{n}{1-\alpha}\frac{\phi^{2}}{2u}\mathcal{L}\left(\frac{u}{\phi}\right). $$
    By the inequality $\frac{\phi^{2}}{2u}\mathcal{L}\left(\frac{u}{\phi}\right)<\frac{D_{\mu}D_{w}}{R}$ , 
    we have our desired result. 
\end{proof}

Recall that the upper bound of $\frac{\Gamma(\sqrt{f})}{f}$ for a vertex $x$ satisfying $\Delta \sqrt{f}(x)<0$ is a crucial step in obtaining the gradient estimate in the $CDE$ condition. Similarly, to obtain the gradient estimate in the $CD\psi$ condition, we need to estimate $\Gamma^{\psi}(f)$ for a vertex $x$ satisfying $\Delta^{\psi}f(x)<0$.  
\begin{lemma}
    Let $\psi:(0,+\infty)\rightarrow \mathbb{R}$ be a $C^{1}$,concave function with $\psi,\ \psi^{\prime}>0$. Suppose $f: V\rightarrow \mathbb{R}$ satisfies $f>0$, and $(\Delta^{\psi} f)(x)<0$ at a vertex $x$. Then 
    $$\Gamma^{\psi}(f)(x)\leq D_{\mu}\left[\psi^{\prime}(1)\left(\psi^{-1}(\psi(1)D_{w})-1\right)+\psi(1)\right]. $$
\label{high}
\end{lemma}
\begin{proof}
    As showed in \cite{wlf}, since $(\Delta^{\psi} f)(x)<0$ and $\psi^{\prime}>0$, we have
    $$\frac{f(y)}{f(x)}\leq\psi^{-1}\left(\psi(1)D_{w}\right),\ \ \ \forall\ y\sim x.$$
    Considering $\psi>0$,  
    we obtain
    \begin{equation*}
    \begin{aligned}
         \Gamma^{\psi}(f)(x)&= \widetilde{\sum_{y\sim x}}\psi^{\prime}(1)\left(\frac{f(y)}{f(x)}-1\right)-\left[\psi\left(\frac{f(y)}{f(x)}\right)-\psi(1)\right]\\
        &\leq \widetilde{\sum_{y\sim x}}\psi^{\prime}(1)\left(\frac{f(y)}{f(x)}-1\right)+\psi(1)\\
        &\leq D_{\mu}\left[\psi^{\prime}(1)\left(\psi^{-1}(\psi(1)D_{w})-1\right)+\psi(1)\right].
    \end{aligned}    
    \end{equation*}
\end{proof}  
\begin{remark}
    Take $\psi=\sqrt{\cdot}$, we obtain $$\frac{\Gamma(\sqrt{u})}{u}\leq \frac{1}{2}D_{\mu}( D_{w}^{2}+1). $$
    However, due to the fact that $D_{w}\geq 1$, the estimate in the $CD\sqrt{\cdot}$ condition will be slightly less refined compared to the estimate in the $CDE$ condition.
\end{remark}

\begin{theorem}
    Let $G=(V,E)$ be a finite graph satisfying the $CD\psi(n,-K)$ for some $K>0$. Let $\psi:(0,+\infty)\rightarrow \mathbb{R}$ be a concave $C^{1}$ function with $\psi,\ \psi^{\prime}>0$. Suppose $u$ is a positive solution of the heat equation \eqref{hq} on $V$. Then 
    $$\Gamma^{\psi}(u)-{\psi}^{\prime}(1)\frac{\partial_{t}u}{u} \leq \frac{n}{2t}+\sqrt{nKC},\ \ \ \forall \ t>0,$$
    where $C=D_{\mu}\left[\psi^{\prime}(1)\left(\psi^{-1}(\psi(1)D_{w})-1\right)+\psi(1)\right].$
\label{deltapsi}
\end{theorem}     

\begin{proof}
    Let $F=t\left[\Gamma^{\psi}(u)-{\psi}^{\prime}(1)\frac{\Delta u}{u}\right]=-t\Delta^{\psi}u$. Fix $T>0$. Let $(x^{*},t^{*})$ be a maximum point of $F$ in $V\times [0,T]$. It suffices to prove that 
    $$F(x^{*},t^{*})\leq \frac{n}{2}+T\sqrt{nKC},$$
    where $C=D_{\mu}\left[\psi^{\prime}(1)\left(\psi^{-1}(\psi(1)D_{w})-1\right)+\psi(1)\right].$ 
    We may assume $F(x^{*},t^{*})>0$. Hence $t^{*}>0$ and $\Delta^{\psi}u(x^{*},t^{*})<0$. In the subsequent analysis, all computations are considered to be carried out at the point $(x^{*},t^{*})$. We apply Lemma \ref{max} with $g=\frac{u}{t^{*}}$. This gives
    \begin{equation*}
        \frac{u}{(t^{*})^{2}}F=\mathcal{L}\left(\frac{u}{t^{*}}\right)F\geq\mathcal{L}\left(\frac{u}{t^{*}}F\right)=2u\Gamma^{\psi}_{2}(u)\geq 2u\left[\frac{1}{n}(\Delta^{\psi}u)^{2}-K\Gamma^{\psi}(u)\right] 
    \end{equation*}
    where we use the $CD\psi(n,-K)$ condition in the last step. 
    Simplifying the equation above, we have 
    $$F\leq 2\left[\frac{1}{n}F^{2}-K(t^{*})^{2}\Gamma^{\psi}(u)\right], $$ 
    which implies
    $$F\leq \frac{n}{2}+t^{*}\sqrt{nK\Gamma^{\psi}(u)}. $$
    By Lemma \ref{high}, our desired result follows. 

\end{proof}     

\begin{theorem}
     Let $G=(V,E)$ be a (finite or infinite) graph satisfying the $CD\psi(n,-K)$ for some $K>0$. Let $\psi:(0,+\infty)\rightarrow \mathbb{R}$ be a concave, $C^{1}$ function with $\psi,\ \psi^{\prime}>0$. Fix $R>0$ and $x_{0}\in V$. Suppose $u$ is a positive solution of the heat equation \eqref{hq} on the ball $B(x_{0},2R)$. Then, in the ball $B(x_{0},R)$,    
    $$\Gamma^{\psi}(u)-{\psi}^{\prime}(1)\frac{\partial_{t}u}{u}\leq \frac{n}{2t}+\sqrt{nKC}+\frac{n D_{\mu}\left[\psi(1)+D_{w}\right]}{R},\ \ \ \forall \ t>0, $$
    where $C=D_{\mu}\left[\psi^{\prime}(1)\left(\psi^{-1}(\psi(1)D_{w})-1\right)+\psi(1)\right].$
\label{psiinf}
\end{theorem}

\begin{proof}
    Let us define a cut-off function $\phi: V \rightarrow \mathbb{R}$ as 
    $$\phi(v)=
    \begin{cases}
    0,& d(v,x_{0})>2R\\ 
    \frac{2R-d(v,x_{0})}{R},& 2R \geq d(v,x_{0})\geq R\\
    1,& R>d(v,x_{0})
    \end{cases}. $$
    Let $F=t\left[\Gamma^{\psi}(u)-{\psi}^{\prime}(1)\frac{\Delta u}{u}\right]=-t\phi\Delta^{\psi}u$. Fix $T>0$. Let $(x^{*},t^{*})$ be a maximum point of $F$ in $V\times [0,T]$. It suffices to prove that 
    $$F(x^{*},t^{*})\leq \frac{n}{2}+T\sqrt{nKC}+T\frac{n D_{\mu}\left[\psi(1)+D_{w}\right]}{R}, $$ 
    where $C=D_{\mu}\left[\psi^{\prime}(1)\left(\psi^{-1}(\psi(1)D_{w})-1\right)+\psi(1)\right].$
    We may assume $F(x^{*},t^{*})>0$. Hence $t^{*}>0$ and $\Delta^{\psi}u(x^{*},t^{*})<0$. In the subsequent analysis, all computations are considered to be carried out at the point $(x^{*},t^{*})$. As in the proof of Theorem \ref{alpha},
    we may assume $\phi(x^{*})\geq \frac{2}{R}$ and $R\geq2$. Hence, $\phi$ does not vanish in the neighborhood of $x^{*}$. Now we apply lemma with $g=\frac{u}{\phi t^{*}}$. This gives
    \begin{equation*}
    \begin{aligned}
        \left[t^{*}\mathcal{L}\left(\frac{u}{\phi}\right)+\frac{u}{\phi}\right]\frac{F}{(t^{*})^{2}}=\mathcal{L}\left(\frac{u}{\phi t^{*}}\right)F&\geq\mathcal{L}\left(\frac{u}{\phi t^{*}}F\right)\\
        &=2u{\Gamma}_{2}^{\psi}(\sqrt{u})\\
        &\geq 2u\left[\frac{1}{n}(\Delta^{\psi}u)^{2}-K\Gamma^{\psi}(u)\right]\\
        &= 2u\left[\frac{F^{2}}{n(t^{*})^{2}\phi^{2}}-K\Gamma^{\psi}(u)\right]
    \end{aligned}    
    \end{equation*}
    where we use the $CD\psi(n,-K)$ condition in the penultimate step. Hence,      
    $$F\leq \frac{n}{2}+t^{*}\sqrt{nK\Gamma^{\psi}(u)}+t^{*}n\left[\frac{\phi^{2}}{2u}\mathcal{L}\left(\frac{u}{\phi}\right)\right].$$
     By $\frac{\phi^{2}}{2u}\mathcal{L}\left(\frac{u}{\phi}\right)<\frac{D_{\mu}D_{w}}{R}$ and Lemma \ref{high}, 
    we have our desired result. 
\end{proof}

\begin{corollary}
   Let $G=(V,E)$ be an infinite satisfying the $CD\psi(n,-K)$ for some $K>0$ with $D_{\mu},D_{w}<\infty$ where $\psi:(0,+\infty)\rightarrow \mathbb{R}$ is a $ C^{1}$, concave function with $\psi,\ \psi^{\prime}>0$. Suppose u is a positive solution of the heat equation \eqref{hq} on $V$. Then  
   $$\Gamma^{\psi}(u)-{\psi}^{\prime}(1)\frac{\partial_{t}u}{u}\leq \frac{n}{2t}+\sqrt{nKC},\ \ \ \forall \ t>0, $$
   where $C=D_{\mu}\left[\psi^{\prime}(1)\left(\psi^{-1}(\psi(1)D_{w})-1\right)+\psi(1)\right].$
\label{Harp}
\end{corollary}

\subsection{Heat kernel bounds}

Now we first recall the Harnack inequality in \cite{sqrt}.
\begin{lemma}\cite{sqrt}
    Let $G=(V,E)$ be a finite graph. Suppose that a function $f:V \times [0,+\infty)\rightarrow\mathbb{R}$ satisfies
    $$(1-\alpha)\frac{\Gamma(f)}{f^{2}}-\frac{\partial_{t}f}{f}\leq\frac{c_{1}}{t}+c_{2}$$
    with some $0<\alpha<1$ and positive constants $c_{1},c_{2}$. Then for $T_{1}<T_{2}$ and $x,y\in V$, we have
    $$f(x,T_{1})\leq f(y,T_{2})\left(\frac{T_{2}}{T_{1}}\right)^{c_{1}}\exp\left\{ c_{2}(T_{2}-T_{1})+\frac{2\mu_{{\rm max}}d^{2}(x,y)}{w_{{\rm min}}(1-\alpha)(T_{2}-T_{1})}\right\}$$
    where $d(x,y)$ denotes the usual graph distance. 
\end{lemma}

Let  $P_{t}(x,y)$ represent the distribution at time $t$ of the continuous-time random walk initiated at $x$. The continuous-time random walk is generated by the walker performing the following action iteratively: waiting for a random duration according to an exponential distribution and then taking a random step. We present, without providing a proof, some widely recognized properties of $P_{t}(x,y)$. 
\begin{proposition}\cite{ran}
    Let $G=(V, E)$ be a finite graph. For any $t>0,\ x,\ y \in V$, the following assertions hold
    \begin{itemize}
        \item[(i)] 
        $\partial_{t}P_{t}(x,y)=\Delta_{x}P_{t}(x,y)=\Delta_{y}P_{t}(x,y)$; 
        
        \item[(ii)] $P_{t}(x,y)\mu(x)=P_{t}(y,x)\mu(y)$; 
        
        \item[(iii)] $\sum_{x\in V}P_{t}(x,y)\mu(x)=1$.      
    \end{itemize}
\label{pro}
\end{proposition}
With these properties, the bounds of $P_{t}(x,y)$ can be derived using the Harnack inequality.
\begin{theorem}
    Suppose the finite graph $G=(V,E)$ satisfies the $CDE(n,-K)$ for some $K>0$. Fix $a>0$. Then we have the following assertions
    \begin{itemize}
        \item[(i)]  For $t>a$, one has
        $$P_{t}(x,y)\geq \frac{C_{3}}{t^{n}}\exp\left\{-C_{1}t-\frac{C_{2}d^{2}(x,y)}{t-a}\right\};$$
        
        \item[(ii)] For $0<t\leq a$, one has 
        $$P_{t}(x,y)\leq \frac{C_{4}\mu(y)\exp\left\{C_{1}t\right\}}{{\rm vol}\left(B(x,\sqrt{t})\right)};$$ 
        
        \item[(iii)] For $t> a$, one has 
        $$\frac{1}{C_{5}{\rm vol}(V)}\left(\frac{t-a}{t}\right)^{n}\leq P_{t}(x,y)\leq\frac{C_{5}}{{\rm vol}(V)}\left(\frac{t+a}{t}\right)^{n} $$      
    \end{itemize}
    where 
    $$C_{1}=\sqrt{2nKD_{\mu}(D_{w}+1)},\ C_{2}=\frac{4\mu_{{\rm max}}}{w_{{\rm min}}},\ C_{3}=a^{n}e^{(C_{1}-1)a},$$
    $$ C_{4}=2^{n}e^{4C_{2}},\ C_{5}=\exp\left\{C_{1}a+\frac{1}{a}C_{2}{\rm diam}^{2}(V) \right\}.$$
    \end{theorem}
\begin{proof}
    Fix $y \in V$. According to Proposition \ref{pro} (i) and Theorem \ref{sf}, for any $x_{1},\ x_{2} \in V$ and $t_{1}<t_{2}$, we have
    \begin{equation}
        P_{t_{1}}(x_{1},y)\leq P_{t_{2}}(x_{2},y)\left(\frac{t_{2}}{t_{1}}\right)^{n}\exp\left\{C_{1}(t_{2}-t_{1})+C_{2}\frac{d^{2}(x_{1},x_{2})}{t_{2}-t_{1}} \right\},
    \label{haa}
    \end{equation}
    where 
    $$C_{1}=\sqrt{2nKD_{\mu}(D_{w}+1)}, \ C_{2}=\frac{4\mu_{{\rm max}}}{w_{{\rm min}}}.$$
    Fix $x\in V$. For any $t>0$, if $z\in B(x,\sqrt{t})$, by \eqref{haa} we obtain  
    $$P_{t}(x,y)\leq P_{2t}(z,y)2^{n}\exp\left\{C_{1}t+4C_{2}\right\}. $$
    Thus for any $t>0$, one has
    \begin{equation*}
    \begin{aligned}
        P_{t}(x,y)&\leq\frac{2^{n}\exp\left\{C_{1}t+4C_{2}\right\}}{{\rm vol}\left(B(x,\sqrt{t})\right)}\sum_{z\in B(x,\sqrt{t})}\mu(z)P_{2t}(z,y)\\
        &\leq \frac{2^{n}\exp\left\{C_{1}t+4C_{2}\right\}}{{\rm vol}\left(B(x,\sqrt{t})\right)}\sum_{z\in B(x,\sqrt{t})}\mu(y)P_{2t}(y,z)\\
        &\leq \frac{2^{n}\mu(y)\exp\left\{C_{1}t+4C_{2}\right\}}{{\rm vol}\left(B(x,\sqrt{t})\right)}
    \end{aligned}
    \end{equation*}
    where we use Proposition \ref{pro} (ii) in the last inequality. 
    Fix $a>0$. Similarly, for any $t>a$, by \eqref{haa}, we have
    $$P_{a}(y,y)\leq P_{t}(x,y)\left(\frac{t}{a}\right)^{n}\exp\left\{C_{1}(t-a)+\frac{C_{2}d^{2}(x,y)}{t-a}\right\}. $$
    Noting that
    $$P_{a}(y,y)\geq \int_{a}^{+\infty} e^{-t}dt=e^{-a},$$
    then (i) follows. 
    
    When $t\rightarrow \infty$, the upper bound above tends to infinite which implies us to find a better one when $t$ is large. Let $t>a$. By Proposition \ref{pro} (iii), there is a point $z\in V$ such that $P_{t+a}(z,y)\leq\frac{1}{{\rm vol}(V)}$. By \eqref{haa}, we have
    $$P_{t}(x,y)\leq P_{t+a}(z,y)\left(\frac{t+a}{t}\right)^{n}\exp\left\{C_{1}a+\frac{1}{a}C_{2}d^{2}(x,z) \right\}$$
    which means
    $$P_{t}(x,y)\leq \frac{1}{{\rm vol}(V)}\left(\frac{t+a}{t}\right)^{n}\exp\left\{C_{1}a+\frac{1}{a}C_{2}{\rm diam}(V)^{2}) \right\}. $$
    Likewise, one has 
    $$P_{t}(x,y)\geq \frac{1}{{\rm vol}(V)}\left(\frac{t-a}{t}\right)^{n}\exp\left\{-C_{1}a-\frac{1}{a}C_{2}{\rm diam}(V)^{2}\right\}.$$

\end{proof}

Subsequently, we generalize the results within the $CD\psi$ condition. The corresponding Harnack inequality is presented below. 
\begin{lemma}\cite{psi}
    Let G=(V,E) be a finite graph, and suppose that a function $f:V \times [0,+\infty)\rightarrow\mathbb{R}$ satisfies
    $$D_{1}\Gamma^{\psi}(f)-\frac{\partial_{t}f}{f}\leq\frac{D_{2}}{t}+D_{3}$$
    along with some positive constants $D_{1},D_{2},D_{3}$. Then for $T_{1}<T_{2}$ and $x,y\in V$ we have
    $$f(x,T_{1})\leq f(y,T_{2})\left(\frac{T_{2}}{T_{1}}\right)^{D_{2}}\exp\left\{ D_{3}(T_{2}-T_{1})+\frac{H_{\psi}d^{2}(x,y)}{D_{1}(T_{2}-T_{1})}\right\}.$$
\end{lemma}

\begin{theorem}
    Suppose the finite graph $G=(V,E)$ satisfies the  $CD\psi(n,-K)$ for some $K>0$ and $\psi:(0,+\infty)\rightarrow \mathbb{R}$ is a $ C^{1}$, concave function with $\psi,\ \psi^{\prime}>0$. Fix $a>0$. Then we have the following assertions
    \begin{itemize}
        \item[(i)]  For $t>a$, one has
        $$P_{t}(x,y)\geq \frac{\overline{C_{3}}}{t^{n}}\exp\left\{-\overline{C_{1}}t-\frac{\overline{C_{2}}d^{2}(x,y)}{t-a}\right\};$$
        
        \item[(ii)] For $0<t\leq a$, one has 
        $$P_{t}(x,y)\leq \frac{\overline{C_{4}}\mu(y)\exp\left\{\overline{C_{1}}t\right\}}{{\rm vol}\left(B(x,\sqrt{t})\right)};$$ 
        
        \item[(iii)] For $t>a$, one has 
        $$ \frac{1}{\overline{C_{5}}{\rm vol}(V)}\left(\frac{t-a}{t}\right)^{\frac{n}{2\psi^{\prime}(1)}}\leq P_{t}(x,y)\leq\frac{\overline{C_{5}}}{{\rm vol}(V)}\left(\frac{t+a}{t}\right)^{\frac{n}{2\psi^{\prime}(1)}} $$      
    \end{itemize}
    where 
    $$\overline{C_{1}}=\frac{1}{\psi^{\prime}(1)}\sqrt{nKD_{\mu}\left[\psi^{\prime}(1)\left(\psi^{-1}(\psi(1)D_{w})-1\right)+\psi(1)\right]}, $$ 
    $$\overline{C_{2}}=H_{\psi}\psi^{\prime}(1),\ \overline{C_{3}}=a^{\frac{n}{2\psi^{\prime}(1)}}e^{(\overline{C_{1}}-1)a},$$
    $$ \overline{C_{4}}=2^{\frac{n}{2\psi^{\prime}(1)}}e^{4H_{\psi}\psi^{\prime}(1)}, \ \overline{C_{5}}=\exp\left\{\overline{C_{1}}a+\frac{1}{a}H_{\psi}\psi^{\prime}(1){\rm diam}^{2}(V) \right\}.$$
\end{theorem}
\begin{proof}
    Fix $y \in V$. According to Proposition \ref{pro} (i) and Theorem \ref{deltapsi}, for any $x_{1},\ x_{2} \in V$ and $t_{1}<t_{2}$, we have
    \begin{equation}
         P_{t_{1}}(x_{1},y)\leq P_{t_{2}}(x_{2},y) \left(\frac{t_{2}}{t_{1}}\right)^{\frac{n}{2\psi^{\prime}(1)}}\exp\left\{\overline{C_{1}}(t_{2}-t_{1})+\frac{H_{\psi}\psi^{\prime}(1)d^{2}(x_{1},x_{2})}{t_{2}-t_{1}} \right\}.
    \label{hab}
    \end{equation}
    where 
    $$\overline{C_{1}}=\frac{1}{\psi^{\prime}(1)}\sqrt{nKD_{\mu}\left[\psi^{\prime}(1)\left(\psi^{-1}(\psi(1)D_{w})-1\right)+\psi(1)\right]}. $$ 
    Fix $x\in V$. For any $t>0$, if $z\in B(x,\sqrt{t})$, by \eqref{hab} we obtain
    $$P_{t}(x,y)\leq P_{2t}(z,y)2^{\frac{n}{2\psi^{\prime}(1)}}\exp\left\{\overline{C_{1}}t+4H_{\psi}\psi^{\prime}(1)\right\}. $$
    Thus for any $t>0$, one has
    \begin{equation*}
    \begin{aligned}
        P_{t}(x,y)&\leq\frac{2^{\frac{n}{2\psi^{\prime}(1)}}\exp\left\{\overline{C_{1}}t+4H_{\psi}\psi^{\prime}(1)\right\}}{{\rm vol}\left(B(x,\sqrt{t})\right)}\sum_{z\in B(x,\sqrt{t})}\mu(z)P_{2t}(z,y)\\
        &\leq \frac{2^{\frac{n}{2\psi^{\prime}(1)}}\exp\left\{\overline{C_{1}}t+4H_{\psi}\psi^{\prime}(1)\right\}}{{\rm vol}\left(B(x,\sqrt{t})\right)}\sum_{z\in B(x,\sqrt{t})}\mu(y)P_{2t}(y,z)\\
        &\leq \frac{2^{\frac{n}{2\psi^{\prime}(1)}}\mu(y)\exp\left\{\overline{C_{1}}t+4H_{\psi}\psi^{\prime}(1)\right\}}{{\rm vol}\left(B(x,\sqrt{t})\right)}
    \end{aligned}
    \end{equation*}
    where we use Proposition \ref{pro} (ii) in the last inequality. 
    Fix $a>0$. Similarly, by \eqref{hab}, we have
    $$P_{a}(y,y)\leq P_{t}(x,y)\left(\frac{t}{a}\right)^{\frac{n}{2\psi^{\prime}(1)}}\exp\left\{\overline{C_{1}}(t-a)+\frac{H_{\psi}\psi^{\prime}(1)d^{2}(x,y)}{t-a}\right\}. $$
    Noting that
    $$P_{a}(y,y)\geq \int_{a}^{+\infty} e^{-t}dt=e^{-a},$$
    then (i) follows. 
    By Proposition \ref{pro} (iii), there is a point $z\in V$ such that $P_{t+a}(z,y)\leq\frac{1}{{\rm vol}(V)}$. By \eqref{hab}, we have
    $$P_{t}(x,y)\leq P_{t+a}(z,y)\left(\frac{t+a}{t}\right)^{\frac{n}{2\psi^{\prime}(1)}}\exp\left\{\overline{C_{1}}a+\frac{1}{a}H_{\psi}\psi^{\prime}(1)d^{2}(x,z) \right\}.$$
    which means
    $$P_{t}(x,y)\leq\frac{1}{{\rm vol}(V)}\left(\frac{t+a}{t}\right)^{\frac{n}{2\psi^{\prime}(1)}}\exp\left\{\overline{C_{1}}a+\frac{1}{a}H_{\psi}\psi^{\prime}(1){\rm diam}^{2}(V) \right\}.$$
    Likewise, one has 
    $$P_{t}(x,y)\geq\frac{1}{{\rm vol}(V)}\left(\frac{t-a}{t}\right)^{\frac{n}{2\psi^{\prime}(1)}}\exp\left\{-\overline{C_{1}}a-\frac{1}{a}H_{\psi}\psi^{\prime}(1){\rm diam}^{2}(V) \right\}.$$

\end{proof}

\begin{remark}
    Take $\psi=\sqrt{\cdot}$, then $$\overline{C_{1}}=\sqrt{2nKD_{\mu}(D_{w}^{2}+1)} , \ \overline{C_{2}}=4,\ \overline{C_{3}}=a^{n}e^{(\overline{C_{1}}-1)a}, $$
    $$\overline{C_{4}}=2^{n}e^{16},\  \overline{C_{5}}=\exp\left\{\overline{C_{1}}a+\frac{4}{a}{\rm diam}(V)^{2} \right\} .$$
\end{remark}

\section{Gradient estimates for a heat-type equation }
\label{4}
In this section, we consider the estimate for the heat-type equation \eqref{yama} on finite graph with the $CD\psi$ condition along with its application for the Harnack inequality. 

\subsection{The $CD\psi$ condition}

The following lemma serves as an alternative to Lemma \ref{L} and shares similarities with the identity presented by Sun in \cite{sun}.

\begin{lemma}
    Let $G=(V,E)$ be a (finite or infinite) graph where $\psi:(0,+\infty)\rightarrow \mathbb{R}$ is a $C^{1}$, concave function with $\psi^{\prime}(1)=0$. Suppose u is a positive solution of the equation \eqref{yama} on $V$. Then 
    $$\mathcal{L}(-u\Delta^{\psi}u)\geq2u\Gamma_{2}^{\psi}(u)+cu^{\sigma}\Delta^{\psi}u. $$    
\end{lemma}

\begin{proof}
    Due to $u$ is a solution of the equation \eqref{yama}, one may derive
    \begin{equation*}
    \begin{aligned}
       \mathcal{L}(-u\Delta^{\psi}u)&=-\Delta(u\Delta^{\psi}u)+\partial_{t}(u\Delta^{\psi}u)\\
        &=-\Delta(u\Delta^{\psi}u)+\left[\Delta u+cu^{\sigma}\right]\Delta^{\psi}u+u\partial_{t}(\Delta^{\psi}u)
    \end{aligned}
    \end{equation*}
   By the definition of $\Omega^{\psi}u$, we show
   \begin{equation*}
    \begin{aligned}
        \partial_{t}(\Delta^{\psi}u)(x)&=\Delta\partial_{t}\left[\psi\left(\frac{u}{u(x)}\right)\right](x)\\
        &=\Delta\left[\psi^{\prime}\left(\frac{u}{u(x)}\right)\cdot\partial_{t}\left(\frac{u}{u(x)}\right)\right](x)\\
        &=\Delta\left[\psi^{\prime}\left(\frac{u}{u(x)}\right)\cdot\frac{u}{u(x)}\cdot\left(\frac{\Delta u+cu^{\sigma}}{u}-\frac{\Delta u(x)+cu^{\sigma}(x)}{u(x)}\right)\right](x)\\
        &=\Omega^{\psi}u(x)+\Delta\left[\psi^{\prime}\left(\frac{u}{u(x)}\right)\cdot\frac{u}{u(x)}\cdot\left(\frac{cu^{\sigma}}{u}-\frac{cu^{\sigma}(x)}{u(x)}\right)\right](x). 
    \end{aligned}
    \end{equation*}
    Since $\psi$ is a $C^{1}$, concave function with $\psi^{\prime}(1)=0$, we have $\psi^{\prime}(s)\geq0, s\in(0,1]$ and $\psi^{\prime}(s)\leq0, s\in[1,+\infty)$. 
    As $c\geq0, \sigma\leq1$ or $c\leq0, \sigma\geq1$, one has
    $$\left[\psi^{\prime}\left(\frac{u}{u(x)}\right)\cdot\frac{u}{u(x)}\cdot\left(\frac{cu^{\sigma}}{u}-\frac{cu^{\sigma}(x)}{u(x)}\right)\right](y)\geq0, \ \ \ \forall\ y\in V.$$
    So this function achieves its minimum in $y=x$. Consequently, 
    $$\Delta\left[\psi^{\prime}\left(\frac{u}{u(x)}\right)\cdot\frac{u}{u(x)}\cdot\left(\frac{cu^{\sigma}}{u}-\frac{cu^{\sigma}(x)}{u(x)}\right)\right](x)\geq0. $$
    Hence, we have
    \begin{equation*}
    \begin{aligned}
        \mathcal{L}(-u\Delta^{\psi}u)&\geq -\Delta(u\Delta^{\psi}u)+\Delta u\Delta^{\psi}u++u\Omega^{\psi}u+cu^{\sigma}\Delta^{\psi}u\\
        &\geq 2u\Gamma^{\psi}_{2}(u)+cu^{\sigma}\Delta^{\psi}u. 
    \end{aligned}
    \end{equation*}
   
\end{proof}

In the subsequent section, we utilize the aforementioned lemma to provide a comprehensive estimate.
\begin{theorem}
    Let $G=(V,E)$ be a finite graph satisfying the $CD\psi(n,-K)$ for some $K>0$ where $\psi:(0,+\infty)\rightarrow \mathbb{R}$ is a $C^{1}$, concave function with $\psi^{\prime}(1)=0$. Suppose $u$ is a positive solution of the equation \eqref{yama} on $V$. 
    Fix $0<\alpha<1$. Then  
    \begin{equation*}
    (1-\alpha)\Gamma^{\psi}u\leq \frac{n}{(1-\alpha)2t}+\frac{Kn}{2\alpha},\ \ \ \forall \ t>0. 
    \end{equation*}
\label{ch}
\end{theorem}   

\begin{proof}
    Let $F=t(1-\alpha)\Gamma^{\psi}u=-t(1-\alpha)\Delta^{\psi}u=-t\left[\Delta^{\psi}u+\alpha\Gamma^{\psi}(u)\right]$. Fix $T>0$. Let $(x^{*},t^{*})$ be a maximum point of $F$ in $V\times [0,T]$. It suffices to prove that 
    $$F(x^{*},t^{*})\leq \frac{n}{2(1-\alpha)}+T\frac{Kn}{2\alpha}.$$ 
    We may assume $F(x^{*},t^{*})>0$. Hence $t^{*}>0$. In the subsequent analysis, all computations are considered to be carried out at the point $(x^{*},t^{*})$. We apply Lemma \ref{max} with $g=\frac{u}{t^{*}}$. This gives
    \begin{equation*}
    \begin{aligned}
        \left[\frac{u}{(t^{*})^{2}}-\frac{cu^{\sigma}}{t^{*}}\right]F=\mathcal{L}\left(\frac{u}{t^{*}}\right)F
        &\geq\mathcal{L}\left(\frac{u}{t^{*}}F\right)\\
        &=(1-\alpha)\mathcal{L}(-u \Delta^{\psi}u)\\
        &\geq2(1-\alpha)u\Gamma^{\psi}_{2}(u)+(1-\alpha)cu^{\sigma}\Delta^{\psi}u\\
        &\geq 2(1-\alpha)u\left[\frac{1}{n}(\Delta^{\psi}u)^{2}-K\Gamma^{\psi}(u)\right]-\frac{cu^{\sigma}}{t^{*}}F\\
    \end{aligned}    
    \end{equation*}
    where we use the $CD\psi(n,-K)$ condition in the last step. 
    Simplifying the inequality above with identity $F=-t^{*}\left[\Delta^{\psi}u+\alpha\Gamma^{\psi}(u)\right]$, we have 
    $$\frac{F}{1-\alpha}\geq 2\left[\frac{1}{n}\left(F+t^{*}\alpha\Gamma^{\psi}(u)\right)^{2}-K(t^{*})^{2}\Gamma^{\psi}(u)\right]. $$
    Let us denote $G=t^{*}\Gamma^{\psi}(u)$. Then we obtain
    $$\frac{F}{1-\alpha}\geq 2\left[\frac{1}{n}(F+\alpha G)^{2}-Kt^{*}G\right]. $$
    After expanding $(F+\alpha G)^{2}$ we throw away the $FG$ term, and use  $\alpha^{2}G^{2}$ to bound the last term $KtG$. So we have 
    $$\frac{F}{1-\alpha}\geq 2\left[\frac{1}{n}F^{2}-\frac{K^{2}n}{4\alpha^{2}}(t^{*})^{2}\right], $$
    which implies
    $$F\leq \frac{n}{2(1-\alpha)}+\frac{Knt^{*}}{\alpha}. $$
    And our desired result follows. 
\end{proof}     
Then we improve the estimate. 
\begin{theorem}
    Let $G=(V,E)$ be a finite graph satisfying the $CD\psi(n,-K)$ for some $K>0$, where $\psi:(0,+\infty)\rightarrow \mathbb{R}$ is a $C^{1}$, concave function with $\psi^{\prime}(1)=0$. Suppose $u$ is a positive solution of the equation \eqref{yama} on $V$.  
    Then   
    $$\Gamma^{\psi}(u) \leq \frac{n}{2t}+Kn,\ \ \ \forall \ t>0. $$
\label{non}
\end{theorem}     

\begin{proof}
    Let $F=t\Gamma^{\psi}u=-t\Delta^{\psi}u$. Fix $T>0$. Let $(x^{*},t^{*})$ be a maximum point of $F$ in $V\times [0,T]$. It suffices to prove that 
    $$F(x^{*},t^{*})\leq \frac{n}{2}+TKn.$$ 
    We may assume $F(x^{*},t^{*})>0$. Hence $t^{*}>0$ . In what follows all computations are understood to take place at the point $(x^{*},t^{*})$. We apply Lemma \ref{max} with $g=\frac{u}{t^{*}}$. This gives
    \begin{equation*}
    \begin{aligned}
        \left[\frac{u}{(t^{*})^{2}}-\frac{cu^{\sigma}}{t^{*}}\right]F=\mathcal{L}\left(\frac{u}{t^{*}}\right)F&\geq\mathcal{L}\left(\frac{u}{t^{*}}F\right)\\
        &=\mathcal{L}(-u \Delta^{\psi}u)\\
        &\geq2u\Gamma^{\psi}_{2}(u)+cu^{\sigma}\Delta^{\psi}u\\
        &\geq 2u\left[\frac{1}{n}(\Delta^{\psi}u)^{2}-K\Gamma^{\psi}(u)\right]+cu^{\sigma}\Delta^{\psi}u\\
        &=2u\left[\frac{1}{n(t^{*})^{2}}F^{2}-\frac{K}{t^{*}}F\right]-\frac{cu^{\sigma}}{t^{*}}F
    \end{aligned}    
    \end{equation*}
    where we use the $CD\psi(n,-K)$ condition in the penultimate step. 
    Hence, we have
    $$F\leq\frac{n}{2}+t^{*}Kn. $$
\end{proof}

\subsection{Harnack inequalities}
By the definition of $H_{\psi}$, one has the following lemma. 
\begin{lemma}\cite{psi}
    Let $\psi:(0,+\infty)\rightarrow \mathbb{R}$ be a $C^{1}$, concave function and $f$ be a positive function on $V$. Then for any $x \in V$, we have
    $$\log{\frac{f(y)}{f(x)}}\leq \sqrt{\frac{H_{\psi}\mu_{{\rm max}}}{w_{{\rm min}}}}\sqrt{\Gamma^{\psi}(f)}(x), \ \ \ \forall \ y\sim x. $$
\label{ln}
\end{lemma}
Now we will give the corresponding Harnack inequality.   
\begin{theorem}
    Let G=(V,E) be a finite graph and $\psi:(0,+\infty)\rightarrow \mathbb{R}$ be a $C^{1}$, concave function with $\psi^{\prime}(1)=0$. Suppose that a function $f:V \times \mathbb{R^{+}}\rightarrow\mathbb{R}$ satisfies
    $$\Gamma^{\psi}(f)\leq\frac{D_{1}}{t}+D_{2}$$
    along with some positive constants $D_{1},D_{2}$. Then we have
    $$f(x,t)\leq f(y,t)\exp\left\{ d(x,y)\sqrt{\frac{H_{\psi}\mu_{{\rm max}}}{w_{{\rm min}}}}\sqrt{\frac{D_{1}}{t}+D_{2}}\right\}, \ \ \ \forall\ x,\ y \in V, \ t>0.$$
\end{theorem}
\begin{proof}
    We firstly assume that $x\sim y$. By Lemma \ref{ln}, we have 
    \begin{equation*}
    \begin{aligned}
        \log{\frac{f(y,t)}{f(x,t)}}\leq \sqrt{\frac{H_{\psi}\mu_{{\rm max}}}{w_{{\rm min}}}}\sqrt{\Gamma^{\psi}(f)(x)}\leq \sqrt{\frac{H_{\psi}\mu_{{\rm max}}}{w_{{\rm min}}}}\sqrt{\frac{D_{1}}{t}+D_{2}}.
    \end{aligned}        
    \end{equation*}
    When $x$ and $y$ are not adjacent, we let path $P=x_{0}\cdots x_{k}$ where $x_{0}=x, x_{k}=y$. Then 
    \begin{equation*}
    \begin{aligned}
        \log{\frac{f(y,t)}{f(x,t)}}=\sum^{k}_{i=1}\log{\frac{f(x_{i},t)}{f(x_{i-1},t)}}
        &\leq\sum^{k}_{i=1}\sqrt{\frac{H_{\psi}\mu_{{\rm max}}}{w_{{\rm min}}}}\sqrt{\frac{D_{1}}{t}+D_{2}}\\
        &=k\sqrt{\frac{H_{\psi}\mu_{{\rm max}}}{w_{{\rm min}}}}\sqrt{\frac{D_{1}}{t}+D_{2}}.
    \end{aligned}
    \end{equation*}
    Choosing a suitable path $P$ such that $k=d(x,y)$, our desired result follows. 
\end{proof}

{\bf Acknowledgements.} Both authors thank to Chuanhuan Li and Kairui Xu for valuable discussion. The first author is supported in part by National Natural Science Foundation of China
(NSFC Grant No. 92066106 and No. 12271093).



\end{document}